\documentclass[12pt,final]{amsproc}

\usepackage[english]{babel}
\usepackage{latexsym}
\usepackage{amssymb}
\usepackage{amscd}
\usepackage[latin1]{inputenc}
\usepackage[mathscr]{eucal}
\usepackage[notcite,notref]{showkeys}

\theoremstyle{plain}
\newtheorem{theorem}{Theorem}
\newtheorem{proposition}{Proposition}

\newtheorem{prob}[theorem]{Problem}

\newtheorem{corollary}{Corollary}
\newtheorem{example}{Example}

\theoremstyle{remark}

\newtheorem{definition}{Definition}
\newcommand{\aprob}{\begin{prob}}
\newcommand{\zprob}{\end{prob}}

\def\Ext{\operatorname{Ext}}

\def\co{\operatorname{co}}

\def\dens{\operatorname{dens}}
\def\N{\mathbb{N}}
\def\K{\mathbb{K}}
\def\R{\mathbb{R}}

\def\vp{\varpi}

\newcommand{\U}{\mathscr{U}}

\title{Stability constants and the homology of quasi-Banach spaces}

\author{F\'elix Cabello S\'anchez}
\address{Departamento de Matem\'{a}ticas, UEx, 06071-Badajoz, Spain}
\email{fcabello@unex.es}

\author{Jes\'us M. F. Castillo}
\address{Departamento de Matem\'{a}ticas, UEx, 06071-Badajoz, Spain}
\email{castillo@unex.es}

\thanks{This research has been supported in
part by project MTM2010-20190-C02-01 and the program Junta de
Extremadura GR10113 IV Plan Regional I+D+i, Ayudas a Grupos de
Investigaci\'on.}

\begin{document}


\bigskip

\bigskip

\maketitle

\begin{abstract}
We affirmatively solve the main problems posed by Laczkovich and
Paulin in \emph{Stability constants in linear spaces}, Constructive
Approximation 34 (2011) 89--106 (do there exist cases in which the
second Whitney constant is finite while the approximation constant
is infinite?) and by Cabello and Castillo in \emph{The long
homology sequence for quasi-Banach spaces, with applications},
Positivity 8 (2004) 379--394 (do there exist Banach spaces $X,Y$
for which $\Ext(X,Y)$ is Hausdorff and non-zero?).
In fact, we
show that these two problems are the same.
\end{abstract}

\thanks{AMS (2010) Class. Number. 41A30, 46B20, 46M15, 46M18,
46A16, 26B25, 52A05, 46B08.}

\section{Introduction}
In this paper we solve the main problems posed \cite{cabecastlong}
and in \cite{lp}. This introductory section is devoted to explain
the problems and the solutions we present. The reader is addressed
to Section~2 for all unexplained terms and notation.

The underlying idea of \cite{cabecastlong} was to study the spaces
$\Ext(X,Y)$ of exact sequences of quasi-Banach spaces $0\to Y \to
Z\to X\to 0$ by means of certain nonlinear maps $F:X\to Y$ called
quasilinear maps. In that paper we introduced (semi-quasi-normed)
linear topologies in the spaces $\Ext(X,Y)$. Such topologies are
completely natural in this setting since they make continuous all
the linear maps appearing in the associated homology sequences. We
encountered considerable difficulties in studying the Hausdorff
character of $\Ext(X,Y)$ for arbitrary spaces $X,Y$, although we
managed to produce some pairs of quasi-Banach spaces $X$ and $Y$
for which $\Ext(X,Y)$ is (nonzero and) Hausdorff. Those examples
invariably required that $\frak L(X,Y)=0$, which is impossible in
a Banach space context. Thus the main problem left open in
\cite[p. 387]{cabecastlong} was to know if there are Banach spaces
$X$ and $Y$ for which $\Ext(X,Y)$ is nonzero and Hausdorff. We
will show in this paper that this is indeed the case.\\

This paper also solves the main  problem in \cite{lp}. In fact, it
was this problem who suggested us the approach of the present
paper and the ``form'' that the examples in Section 6 had to have.
In  \cite{lp} Laczkovich and Paulin studied the stability of
approximately affine and Jensen functions on convex sets of linear
spaces. They show that, in many respects, it is enough to work the
case where the convex set is the unit ball of a Banach space.
Reduced to its basic elements the main problem in \cite{lp} can be
stated as follows. Let $B$ a convex set and $Y$ a Banach space. A
function $f:B\to Y$ is approximately affine if
$$
\|f(tx+(1-t)y)-tf(x)-(1-t)f(y)\|_Y\leq 1
$$
for every $x,y\in B$ and every $t\in[0,1]$. Suppose there is a
constant $A_0=A_0(B,Y)$ such that for every \emph{bounded}
approximately affine function $f:B\to Y$ there is a true affine
function $a:B\to Y$ such that $\|f(x)-a(x)\|\leq A_0$ for every
$x\in B$. Does it follow that for every (in general unbounded)
approximately affine function $f:B\to Y$ there is an affine
function $a$ such that $\sup_{x\in B}\|f(x)-a(x)\|<\infty$?

We show in Section 8 that the Examples in Section 6 provide a
negative answer to this problem. The idea is that if $B$ is the
unit ball of $X$, then every approximately affine function $f:B\to
Y$  is (up to a bounded perturbation) the restriction of some
quasilinear map $F:X\to Y$. Thus  $A_0(B,Y)$ and $K_0(X,Y)$ are
nearly proportional and so, it suffices to get a pair of Banach
spaces with $K_0(X,Y)$ finite but admitting a  nontrivial
quasilinear map $F:X\to Y$.\\

Let us describe the organization of the paper. Section~2 is
preliminary: it contains the necessary background on quasilinear
maps and extensions. In Section 3 we give a seemingly new
representation of $\Ext(X,Y)$, we denote by $\Ext^\vp(X,Y)$, which
is based on a relatively projective presentation of $X$ and
carries a natural topology. Section 4 is devoted to show that the
``old'' functor $\Ext$ and the ``new'' functor $\Ext^\vp$ are
naturally equivalent, both algebraically and topologically. In
Section 5 we give the basic criterion to determine when
$\Ext(X,Y)$ is Hausdorff. It turns out that $\Ext(X,Y)$ is
Hausdorff if and only if a certain parameter $K_0(X,Y)$ (depending
only on the behavior of the bounded quasilinear maps $F:X\to Y$)
is finite. In Section 6 we arrive to the main counter-examples: we
show that $\Ext(X,Y)$ is nonzero and Hausdorff if, for instance,
$X=\ell_p(I)$ with $1<p<\infty$ and $I$ uncountable and $Y=c_0$ is
the space of all null sequences with the sup norm. The results of
Section~7 confirm that $\Ext(X,Y)$ has a quite strong tendency to
not being Hausdorff.

The paper has been organized so that the interested reader can go
straight to the solution of Laczkovich-Paulin problem. The
shortest path is to read first Sections~\ref{qm} and \ref{ts},
then Section~\ref{count} and finally Section~\ref{affine}.

\section{Background on quasilinear maps and extensions of quasi-Banach spaces}

\subsection{Notation}
Let $X$ be a (real or complex) linear space. A semi quasi-norm is
a function $\varrho:X\to \R_+$ satisfying the following
conditions:
\begin{itemize}
\item[(a)] $\varrho(\lambda x)=|\lambda|\varrho(x)$ for every
$x\in X$ and every scalar $\lambda$. \item[(b)] There is a
constant $\Delta$ such that $\varrho(x+y)\leq
\Delta(\varrho(x)+\varrho(y))$ for every $x,y\in X$.
\end{itemize}
A semi quasi-normed space is a linear space furnished $X$ with a
semi quasi-norm $\varrho$ which gives rise to a linear topology,
namely the least linear topology for which the set $B_X=\{x\in X:
\varrho(x)\leq 1\}$ is a neighborhood of the origin. A semi
quasi-Banach space is a complete semi quasi-normed space. If
$\varrho(x)=0$ implies $x=0$, then $\varrho$ is said to be a
quasi-norm and $X$ is Hausdorff.

Throughout the paper we denote by $\Delta_\varrho$ (or $\Delta_X$
if there is no risk of confusion) the modulus of concavity of
$(X,\varrho)$, that is, the least constant $\Delta$ for which (b)
holds.

Let $X$ and $Y$ be quasi-Banach spaces. A mapping $F:X\to Y$ is
said to be homogeneous if $F(\lambda x)=\lambda F(x)$ for every
scalar $\lambda$ and every $x\in X$. A homogeneous map $F$ is
bounded if there is a constant $C$ such that $\|F(x)\|_Y\leq
C\|x\|_X$ for all $x\in X$; equivalently if it is uniformly
bounded on the unit ball. We write $\|F\|$ for the least constant
$C$ for which the preceding inequality holds. Clearly,
$\|F\|=\sup_{\|x\|\leq 1}\|F(x)\|$. This is coherent with the
standard notation for the (quasi-) norm of a linear operator.

\subsection{Quasilinear maps}\label{qm}
Let $X$ and $Y$ be quasi-Banach spaces. A map $F:X\to Y$ is said to be quasilinear if it is
homogeneous and satisfies an estimate
\begin{equation}\label{q}
\|F(x+y)-F(x)-F(y)\|_Y\leq Q\left(\|x\|_X+\|y\|_X\right)
\end{equation}
for some constant $Q$ and all $x,y\in X$. The smallest constant
$Q$ above shall be denoted by $Q(F)$.

All linear maps, continuous or not, are quasilinear and so are the bounded (homogeneous) maps.

Let us say that a quasilinear map $F$ is trivial if it is the sum of a linear map $L$ and a bounded homogeneous map $B$.
This happens if and only $F$ is at finite ``distance'' from $L$ in the sense that $\|F-L\|<\infty$.

Let us introduce the ``approximation constants'' for quasilinear maps as follows.

\begin{definition}
Given quasi-Banach spaces $X$ and $Y$ we denote by $K_0(X,Y)$ the
infimum of those constants $C$ for which the following statement
is true: if $F:X\to Y$ is a \emph{bounded} quasilinear map, then
there is a linear map $L:X\to Y$ such that $\|F-L\|\leq CQ(F)$.

We define $K(X,Y)$ analogously just omitting the word `bounded'.
\end{definition}

Observe that $K_0(X,Y)$ does not vary if we replace ``bounded'' by ``trivial'' in the definition.

\subsection{Twisted sums}\label{ts}
The link between quasilinear maps and the homology of (quasi-) Banach spaces is the following construction, due to Kalton~\cite{kalton78} and Ribe \cite{ribe}.
Suppose $F:X\to Y$ is quasilinear. We consider the product space $Y\times X$ and we furnish it with the following quasinorm:
$$
\|(y,x)\|_F=\|y-F(x)\|_Y+\|x\|_X.
$$
Following a long standing tradition, we denote the resulting quasi-Banach space by $Y\oplus_F X$. It is obvious that the operator $\imath:Y\to Y\oplus_F X$ sending $y$ to $(y,0)$ is an isometric embedding, so we may regard $Y$ as a subspace of  $Y\oplus_F X$. The connection between $X$ and $Y\oplus_F X$ is less obvious. Consider the operator $\pi:Y\oplus_F X\to X$ defined by $\pi(y,x)=x$.
We have $\|\pi\|\leq 1$, by the very definition of the quasi-norm in $Y\oplus_F X$. Besides,
 $\pi$ maps the unit ball of $Y\oplus_F X$ onto the unit ball of $X$ since $\pi(F(x),x)=x$ and $\|(F(x),x)\|_F=\|x\|_X$. Hence $\pi$ is a quotient map and, moreover, its kernel equals the image of $\imath$, so $X=(Y\oplus_F X)/\imath(Y)$. Most of the preceding information can be rephrased by saying that the diagram
\begin{equation}\label{yfx}
\begin{CD}
0 @>>>Y @>\imath>> Y\oplus_F X@>\pi>> X @>>> 0
\end{CD}
\end{equation}
is a short exact sequence of quasi-Banach spaces (``exact'' means that the kernel of each arrow
coincides with the image of the preceding one), or else, that $Y\oplus_F X$ is a twisted sum of $Y$ and $X$.

One may wonder under which conditions the sequence (\ref{yfx})
splits, that is, there is an operator $p: Y\oplus_F X \to Y$ such
that $p\circ \imath={\bf I}_Y$ (equivalently, the subspace
$\imath(Y)$ is complemented in $Y\oplus_F X$ through $\imath\circ
p$). The answer is quite simple: (\ref{yfx}) splits if and only if
$F$ is trivial. Indeed, every linear map, continuous or not, from
$Y\oplus_F X$ to $Y$ can be written as $p(y,x)=M(y)-L(x)$, where
$M:Y\to Y$ and $L:Z\to Y$ are linear. Moreover, $p\circ\imath={\bf
I}_Y$ if and only $M={\bf I}_Y$. Let us see that the
``projection'' $p(y,x)=y-L(x)$ is bounded if and only if $L$ is at
finite distance from $F$. Indeed, if $p$ is bounded, we have
$$
\|y-L(x)\|_Y=\|p(y,x)\|_Y\leq \|p\|\|(y,x)\|_F=\|p\|(\|y-F(x)\|_Y+\|x\|_X),
$$
and taking $y=F(x)$ we see that $\|F-L\|\leq \|p\|$. The converse is also easy:
\begin{align*}
\|p(y,x)\|_Y=\|y-L(x)\|_Y&\leq \Delta_Y(\|y-F(x)\|_Y+\|F(x)-L(x)\|_Y)\\
&\leq  \Delta_Y(\|y-F(x)\|_Y+\|F-L\|\|x\|_X)
\end{align*}
and so $\|p\|\leq \Delta_Y\max(1,\|F-L\|)$.

Thus quasilinear maps give rise to twisted sums (exact sequences). The converse is also true. Indeed if we are given a pair of quasi-Banach spaces $X$ and $Y$ and a third space $Z$ containing $Y$ in such a way that $Z/Y=X$, then we can define a quasilinear map from $X$ to $Y$ as follows. First, we take a homogeneous bounded section of the quotient map $\pi:Z\to Y$, that is, a homogeneous map $B:X\to Z$ such that if $z=B(x)$, then $\pi(z)=x$, with $\|z\|_Z\leq 2\|x\|_X$ (note that $\|x\|_X=\inf_{\pi(z)=x}\|z\|_Z$). Of course $B$ will be not linear nor continuous in general. Then we may take a linear (but probably unbounded) section of $\pi$. Finally, we put $F=B-L$. Note that $F$ takes values in $Y$; moreover it is really easy to check that $F$ is quasilinear since for $x,y\in X$ one has
$$
F(x+y)-F(x)-F(y)= B(x+y)-B(x)-B(y).
$$
It turns out that $Z$ is linearly isomorphic to the twisted sum $Y\oplus_F X$ through the map $u:Z\to Y\oplus_F X$ given by $u(z)=(z-L(\Pi(z)), \pi(z))$ and so $F$ is trivial if and only if $Y$ is complemented in $Z$, in which case $Z$ is isomorphic to $Y\oplus X$.

\subsection{Extensions}
Let $X$ and $Y$ be quasi-Banach spaces.
An extension of $X$ by $Y$ is a short exact sequence of quasi-Banach spaces and operators
\begin{equation}\label{SEX}
\begin{CD}
0 @>>>Y @>\imath>>Z@>\pi>> X @>>> 0.
\end{CD}
\end{equation}
 As $\imath(Y)=\ker\pi$ is closed in $Z$ the operator $\imath$ embeds $Y$ as a closed subspace of $Z$ and $X$ is
isomorphic to the quotient $Z/\imath(Y)$, by the open mapping
theorem. The extension $0
\to Y \to Z' \to X\to 0$ is said to be equivalent to (\ref{SEX})
if there exists an operator $t: Z \to Z'$ making commutative the
diagram
\begin{equation}\label{equiv}
\begin{CD}
0  @>>>  Y @>>> Z @>>> X @>>>0 \\
  &  &   @|  @VVtV  @|
   & \\
0  @>>>  Y @>>> Z_1 @>>> X @>>>0. \\
\end{CD}
\end{equation}
The five-lemma \cite[Lemma~3.3]{maclane} and the open mapping theorem imply that such a $t$ is necessarily an isomorphism and so ``being equivalent'' is a true equivalence relation. The extension (\ref{SEX}) is
said to be trivial if it is equivalent to the direct sum $0\to Y\to Y\oplus X\to X\to 0$. This happens if and only if it
splits, that is, there is an operator $p:Z\to Y$ such that
$p\circ\imath={\bf I}_Y$ (i.e., $\imath(Y)$ is complemented in $Z$); equivalently,
there is an operator $s:X\to Z$ such that $\pi \circ s= {\bf I}_X$.

For every pair of quasi-Banach spaces $X$ and $Y$, we denote by
$\Ext(X,Y)$ the space of all exact sequences (\ref{SEX}) modulo
equivalence. The set $\Ext(X,Y)$ can be given a linear structure
in such a way that the class of trivial sequences corresponds to
zero; see \cite[Appendix~7]{cabecastlong}. Thus, $\Ext(X,Y)=0$
means that every extension of the form (\ref{SEX}) is trivial. By
\cite[Proposition 3.3]{kalton78}) $K(X,Y)$ is finite if and only
if $\Ext(X,Y)=0$.

There is a correspondence between extensions and quasilinear maps
(sketched in Section~\ref{ts} and considered in full detail in
\cite{kaltpeck,castgonz} and \cite[Section~2]{cabecastlong}) which
takes \emph{trivial} quasilinear maps into  \emph{trivial}
extensions. Moreover, two quasilinear maps induce equivalent
extensions if and only if its difference is trivial in which case
we will declare it as ``equivalent''. From now on we will denote
by $\mathcal Q(X,Y)$ the space of quasilinear maps from $X$ to $Y$
modulo equivalence.

\section{A new construction of $\Ext$}

Two  ``representations'' of the spaces $\Ext(X,Y)$ are available so far, namely, the very definition through short exact sequences and the identification with $\mathcal Q(X,Y)$ provided by the construction that appears in Section~\ref{ts}.
To get the results of this paper  we will
need a new one. To explain it, observe that when one works in
Banach spaces things are somewhat simpler since there exist
projective objects ($\ell_1(I)$ for every set of indices $I$) and injective objects
($L_\infty(\mu)$-spaces, for instance). In the category {\bf
Q} of quasi-Banach spaces each projective object is
finite-dimensional. This easily follows from the fact, proved by
Stiles in \cite{stiles}, that every infinite-dimensional
complemented subspace of $\ell_p$ is isomorphic to $\ell_p$ when
$0<p<1$. (See also \cite[Chapter 2, Section 3]{kpr}.) This could
be a problem for the construction of $\Ext$ via projective
objects. However, there exist spaces that act as projective, at
least for a given couple of quasi-Banach spaces. We base our
approach on two facts: the first one is Aoki-Rolewicz theorem that states that every quasi-Banach space is $p$-normable for some $p\in(0,1]$; see \cite[Theorem 3.2.1]{r} or
\cite[Theorem 1.3]{kpr}.

The second fact, proved by Kalton in \cite[Theorem
3.5]{kalton78} is that $\Ext(\ell_q(I), Y)=0$ if $0<q<1$ and $Y$ is
$p$-normable for some $p>q$. A `formal' consequence is that any
twisted sum of two $p$-normable spaces is $q$-normable for every
$0<q<p$ (see \cite[Theorem 4.1]{kalton78} for a more general
result). Now, let $X$ and $Y$ be fixed quasi-Banach spaces, both
$p$-normable. Take $0<q<p$ and let $I$ be a set of indices having
cardinality $\dens(X)$ so that we can construct a quotient map
$\varpi: \ell_q(I)\to X$. Write $K=\ker \varpi$ and consider the
extension
\begin{equation}\label{pre}
0 \longrightarrow K \stackrel \varkappa \longrightarrow \ell_q(I)  \stackrel \varpi \longrightarrow X \longrightarrow 0.
\end{equation}
Applying $\mathfrak L(-,Y)$ we get the exact sequence (see Theorem
1 in \cite{cabecastlong})
$$
0 \longrightarrow \mathfrak L(X, Y) \stackrel {\varpi^*}
\longrightarrow \mathfrak  L(\ell_q(I), Y) \stackrel {\varkappa^*}
\longrightarrow \mathfrak  L(K, Y) \stackrel  \omega
\longrightarrow \Ext(X, Y) \longrightarrow 0
$$
since $\Ext(\ell_q(I), Y)=0$. The exactness of the preceding
sequence encodes the following information:

\begin{itemize}
\item Each operator $u:K\to Y$ gives rise to an extension of $X$
by $Y$, namely the lower row in the commutative diagram:
$$
\begin{CD}
0@>>> K@>>> \ell_q(I) @>>> X @>>>0\\
& & @V u VV @VVV @|\\
0@>>> Y @>>> P @>>> X @>>>0.
\end{CD}
$$(see Section 7 in \cite{cabecastlong})

\item Two operators $u,v\in \mathfrak
L(K,Y)$ give equivalent extensions if and only if $u-v$
extends to an operator $\mathfrak L(\ell_q(I), Y)$.

\item All extensions $0\to Y\to Z\to X\to 0$ arise in this
way, up to equivalence.
\end{itemize}

The third point is as follows: if $0\to Y\to Z\to X\to 0$ is exact, then $Z$ is
$q$-normable and the quotient map $\varpi:\ell_q(I)\to X$ can be
lifted to an operator $u: \ell_q(I)\to Z$. The restriction of $u$
to $K$ takes values in $Y$ and therefore one gets a commutative
diagram
$$
\begin{CD}
0@>>> K @>>> \ell_q(I) @>>> X @>>>0\\
& & @V u VV @VVV @|\\
0@>>> Y @>>> Z @>>> X @>>>0.
\end{CD}
$$
Accordingly we can define a linear space taking
$$
\Ext^\varpi(X,Y)= \frac{\mathfrak L(K, Y)}{\varpi^* (\mathfrak
L(\ell_q(I), Y))}.
$$

The algebraic part of the equivalence between the three
representations: $\Ext$ (equivalence classes of exact sequences),
 $\mathcal Q$ (equivalence classes of quasilinear
maps) and
$\Ext^\varpi$ are more or less straightforward and, as soon as one
gets some acquaintance with the functor and natural transformation
language \cite[Chapter I, Section~8]{maclane} can be formulated as:

\begin{proposition} The functors
$\mathcal Q, \Ext_{\bf Q}$ and $\Ext^\varpi$ are naturally
equivalent acting from the category $\bf Q\times\bf Q^{op}$ to the
category {\bf V} of vector spaces.\hfill$\square$
\end{proposition}

\section{Natural equivalences for $\Ext$}

More interesting, and essential for our purposes, is to consider
what occurs when one endows the spaces with natural linear
topologies. The space $\mathcal Q(X,Y)$ can be equipped (see
Section 3 of \cite{cabecastlong}) with the semi-quasi-norm
$$
Q([F])=\inf\{Q(G): G \text{ is equivalent to } F\}.
$$
It was proved in \cite[Theorem~4]{cabecastlong} that this makes $\mathcal
Q(X,Y)$ complete, but rarely Hausdorff.  The corresponding
topology in $\Ext(X,Y)$ can be obtained as follows. Given an
extension
\begin{equation}\label{given}
0 \longrightarrow Y \stackrel \imath \longrightarrow Z  \stackrel \pi \longrightarrow X \longrightarrow 0,
\end{equation}
we put
$$
\varrho(\imath,\pi)=\inf_B\sup_{x,y\in X}\frac{\|\imath^{-1}(B(x+y)-Bx-By)\|_Y}{\|x\|_X+\|y\|_X},
$$
where $B$ runs over all homogeneous bounded sections $B:X\to Z$
of $\pi$. Then set
$$
\varrho([(\imath,\pi)])=\inf_{(\imath',\pi')} \varrho(\imath',\pi'),
$$
where the infimum is taken over those $(\imath',\pi')$ for which
$$
0 \longrightarrow Y \stackrel {\imath'} \longrightarrow Z'  \stackrel {\pi'} \longrightarrow X \longrightarrow 0
$$
is equivalent to (\ref{given}). But the representation we need for
$\Ext$ is  $\Ext^\varpi$. Let us introduce a semi-quasi-norm on
$\Ext^\varpi(Z,Y)$ as:
$$
|[u]|_\varpi=\inf\left\{\|u-\varkappa^*(v)\|_{\mathfrak L(K,Y)}:
v\in \mathfrak  L(\ell_q(I),Y)\right\}
$$
The image of $\varkappa^*$ is not (as a rule) closed in $\mathfrak
 L(K, Y)$ and thus the topology induced by $ |\cdot|_\varpi$ is not
necessarily Hausdorff, which matches with the previous
constructions. One has

\begin{theorem}
\label{main} $\Ext, \mathcal Q$ and $\Ext^\varpi$ are naturally
equivalent functors acting from the category $\bf Q\times\bf
Q^{op}$ to the category $\bf Q_{1/2}$ of semi quasi-Banach spaces.
\end{theorem}
\begin{proof} That $\Ext$ and $\mathcal Q$ still are naturally equivalent when
regarded as functors $\bf Q\times\bf Q\rightsquigarrow \bf \bf
Q_{1/2}$ is nearly obvious. So let us concentrate on the natural
equivalence between $\mathcal Q$ and $\Ext^\varpi$. Given
quasi-Banach spaces $X$ and $Y$ we fix a (relatively projective)
presentation for $X$, as in (\ref{pre}). Let $B_0$ and $L_0$
represent, respectively, a homogeneous bounded and a linear
selection for $\varpi$ so that $\Phi=B_0-L_0$ is a quasilinear
map associated to (\ref{pre}). We define a linear map from
$\frak L(K,Y)$ to the space of quasilinear maps $Z\to Y$ by
$\Phi^*(u)=u\circ\Phi$. We have $Q(u\circ\Phi)\leq \|u\|Q(\Phi)$.
Suppose $u$ extends to an operator $v: \ell_q(I) \to Y$. Then the
identity
$$u\circ\Phi = v\circ B_0-v\circ L_0$$
shows that $[u\circ\Phi]=0$ in $\mathcal Q(X,Y)$. Thus we can
define $\eta$ at the couple $(X,Y)$ by
$$\eta_{(X,Y)} [u] = [u\circ\Phi]\quad\quad(u\in L(K,Y))$$
From now on, assume that $Y$ and $Z$ have been fixed so that we
can omit the subindex and just write $\eta$. Quite clearly
$$
Q[\eta([u])]\leq Q(\Phi)|[u]|_\varpi,
$$
and so $\eta:(\Ext^\varpi(X,Y),|\cdot|_\varpi)\to(\mathcal
Q(X,Y),Q[\cdot])$ is (linear and) continuous. From now on we will
identify $\ell_q(I)=K\oplus_\Phi X$ through the isomorphism
$z\longmapsto(z+L_0(\vp(z)), \vp(z))$ without further mention. Let
us check the injectivity of $\eta$. Given $u\in L(K,Y)$ the
following diagram is commutative:
$$
\begin{CD}
0@>>> K@>>> K\oplus_\Phi X @>\varpi >> X @>>>0\\
& & @V u VV @V U VV @|\\
0@>>> Y @>>> Y\!\oplus_{\Phi^*(u)}\!X @>\pi >> X @>>>0
\end{CD}
$$
where $U(k,x)=(u(k),x)$. If  $u\circ\Phi$ is trivial, composing $U$
with a bounded projection of $Y\oplus_{\Phi^*(u)} Z $ onto $Y$ we get an
extension of $u$ to $\ell_q(I)$ and so $[u]=0$ in $\Ext^\vp(X,Y)$.
Next we show $\eta$ is onto. If $F: X\to Y$ is quasilinear,
consider the induced exact sequence $0\to Y\to Y\oplus_F X\to X\to
0$. Let $U:\ell_q(I)\to Y\oplus_F X$ be any operator lifting the
quotient $Y\oplus_F X\to X$. Clearly $u$, the restriction of $U$
to $K$, takes values in $Y$ and we have a commutative diagram
\begin{equation}\label{dia}
\begin{CD}
0@>>> K@>>> \ell_q(I) @>\varpi >> X @>>>0\\
& & @V u VV @V U VV @|\\
0@>>> Y @>>> Y\oplus_F X @>\pi >> X @>>>0
\end{CD}
\end{equation}
which shows that $\eta([u])=[F]$. Should you need some more explanations, please identify $\ell_q(I)=K\oplus_\Phi X$ as indicated before. After that, $U$ has the form $U(k,x)=(u(k)+\ell(x),x)$, where $u\in L(K,Y)$ and $\ell:X\to Y$ is a linear (but maybe discontinuous) map. Anyway $U$ is bounded, so
$$
\|u(k)+\ell(x)-F(x)\|_Y+\|x\|_X\leq \|U\| (\|k-\Phi(x)\|_K+\|x\|_X).
$$
Taking $k=\Phi(x)$ we have
$$
\|u(\Phi(x))+\ell(x)-F(x)\|_Y+\|x\|_X\leq \|U\| \|x\|_X,
$$
hence $\|F-(u\circ \Phi+ \ell)\|\leq \|U\|-1$ and in particular
$[F]=[u\circ \Phi]=\eta[u]$.  We have thus seen that $\eta$ is a
continuous linear bijection. The preceding argument shows a bit
more: given a quasilinear $F:X\to Y$, $\eta^{-1}([F])$ is the
(class in $\Ext^\vp(X,Y)$ of) the restriction to $K$ of any
operator $U:\ell_q(I)\to  Y\oplus_F X$ lifting the quotient
$Y\oplus_F X\to X$. As $\eta^{-1}:\mathcal Q(X,Y)\to
\Ext^\vp(X,Y)$ is linear ($\eta$ is), continuity will follow if we
show that there is a constant $C$ such that, given $F:X\to Y$ with
$Q(F)\leq 1$, there is a lifting $U:\ell_q(I)\to  Y\oplus_F X$
with $\|U\|\leq C$. We will prove that the obvious lifting does
the work ---see (\ref{U}) below.

To this end let us remark that, if $(z_i)$ converges to $z$ in the
quasi-normed space $Z$, then $\|z\|\leq \Delta_Z \limsup_i\|z_i\|$,
where $\Delta_Z$ is the modulus of concavity of $Z$.

Next, if $F:X\to Y$ is quasilinear,
then the modulus of concavity of $Y\oplus_F
X$ is at most $\max\{\Delta_Y^2, \Delta_Y Q(F)+\Delta_X\}$.
Indeed, one has
\begin{align*}
\|(y&+y',x+x')\|_F=\|y+y'-F(x+x')\| + \|x+x'\|\\
&\leq
\Delta_Y(\|y+y'-F(x)-F(x')\|+\|F(x)+F(x')-F(x+x')\|)+\|x+x'\|\\
&\leq
\Delta_Y(\Delta_Y(\|y-F(x)\|+\|y'-F(x')\|)+Q(F)(\|x\|+\|x'\|)+ \Delta_X(\|x\|+\|x'\|)\\&
\leq \max\{\Delta_Y^2, \Delta_Y Q(F)+\Delta_X\}(\|(y,x)\|_F+\|(y',x')\|_F).
\end{align*}

We have mentioned that extensions of $p$-normable spaces are
$q$-normable for $0<q<p$. On the other hand a quasilinear map
$F$ gives rise to a $q$-normable extension if and only if it
satisfies the estimate
\begin{equation*}
\left\|F\left(\sum_{i=1}^n x_i\right)-\sum_{i=1}^n F(x_i)\right\|\leq M \left(\sum_{i=1}^n\|x_i\|^q\right)^{1/q}
\end{equation*}
for some $M$ and all $n$, with $x_i$ in the domain of $F$.
An obvious ``amalgamation'' argument shows the existence of a constant $M=M(p,q)$ such that, if $F$ is a quasilinear map acting between $p$-normed spaces, $0<q<p$, and $x_1,\dots,x_n$ belong to the domain of $F$, then
$$
\left\|F\left(\sum_{i=1}^n x_i\right)- \sum_{i=1}^n F(x_i)\right\|\leq M Q(F)\left(\sum_{i=1}^n\|x_i\|^q\right)^{1/q}.
$$

After that, suppose we are given a quasilinear map $F: X\to Y$. We define a lifting
$U:\ell_q(I)\to Y\oplus_F X$ (of the) quotient map $\varpi:\ell_q(I)\to X$) through the formula
\begin{equation}\label{U}
U\left(\sum_{i\in I} \lambda_ie_i\right) =  \sum_{i\in I}
\lambda_i (F\varpi e_i, \varpi e_i).
\end{equation}
The summation of the right-hand side of (\ref{U}) is performed in the quasi-norm topology of $Y\oplus_F Z$. The series converges because $Y\oplus_F Z$ is $q$-normable and complete and
$$
\sum_i\|\lambda_i (F\varpi e_i, \varpi e_i)\|_F^q=\sum_i|\lambda_i|^q<\infty.
$$
Let us estimate $\|U\|$ assuming $Z$ and $Y$ are $p$-normed and $Q(F)\leq 1$.
\begin{align*}
\left\|U\left(\sum_i\lambda_ie_i\right)\right\|_F&= \left\| \sum_i
\lambda_i (F\varpi e_i, \varpi e_i)\right\|_F\\
&\leq\Delta\cdot \sup \left\{ \left\| \sum_{i\in J}
\lambda_i (F\varpi e_i, \varpi e_i)\right\|_F: J\text{ is a finite subset of }I \right\}\\
&= \Delta\cdot \sup_{J}\left\{\left\| \sum_{i\in J}
\lambda_i (F\varpi e_i)-F\left(\sum_{i\in J}\lambda_i\varpi e_i \right)\right\|_Y+ \left\| \sum_{i\in J}\lambda_i\varpi e_i \right\|_X  \right\}\\
&\leq  \Delta\cdot \sup_{J} \left\{  M \left( \sum_{i\in J} |\lambda_i|^q\right)^{1/q}       +      \left( \sum_{i\in J} |\lambda_i|^q\right)^{1/q}                  \right\}\\
&\leq \Delta (M+1)\left\|\sum_i \lambda_i e_i\right\|_q,
\end{align*}
where $\Delta\leq\max(\Delta_Y^2, \Delta_Y+\Delta_X)$ is the modulus of concavity of $Y\oplus_F X$.
To sum up, we have seen that if $Q(F)\leq 1$ then $
\|U\|\leq (M(p,q)+1) \max(\Delta_Y^2, \Delta_Y+\Delta_X) $ and the same bound holds for $u=U|_K$. As $\eta([u])=[F]$ and taking into account that $\eta$ is linear, in particular homogeneous, we have
$$
\eta^{-1}([F])\leq CQ([F]),
$$
 where $C=(M(p,q)+1) \max(\Delta_Y^2, \Delta_Y+\Delta_X)$.
This completes the proof.
\end{proof}

It is perhaps worth noticing that the correspondences $F\mapsto U$
and $F\mapsto u$ described in the preceding proof are not linear
(or even homogeneous): $F\mapsto u$ becomes linear only when one
passes to the quotient structures.

\section{Approximation constants and the Hausdorff character of $\Ext$}

\begin{theorem}\label{criterion}
Let $X$ and $Y$ be quasi-Banach spaces. Then $\Ext(X,Y)$ is
Hausdorff if and only if $K_0(X,Y)$ is finite.
\end{theorem}
\begin{proof} An operator $t:A\to B$ acting between quasi-Banach spaces has
closed range if and only if it is relatively open, that is, there
is a constant $C$ such that whenever $b=t(a)$ for some $a\in A$
there is $a'\in A$ such that $b=t(a')$ and $\|a'\|\leq C\|b\|$.

We will work with the space $\Ext^\varpi(X, Y)$ arising from a
relatively projective presentation of $X$, as in (\ref{pre}). As
we have already shown, this space and  $\Ext(X,Y)$ (or $\mathcal
Q(X,Y)$) are linearly isomorphic. Looking at the exact sequence
$$
\begin{CD}
0 @>>> \mathfrak L(X, Y) @>{\varpi^*}>> \mathfrak L(\ell_q(I), Y)
@>{\varkappa^*}>> \mathfrak L(K, Y) @>{\omega}>> \Ext^\varpi(X, Y)
@>>> 0
\end{CD}
$$we see that $\Ext^\varpi(Z,Y)$ is Hausdorff if and only if the
(restriction) map $\varkappa^*$ has closed range. By the preceding
remarks this is equivalent to $\varkappa^*$ being relatively open,
namely:
\begin{itemize}
\item[$\bullet$] There is a constant $C$ such that whenever $u:K\to Y$ can be
extended to  $\ell_q(I)$ there is an extension $\tilde u\in
\frak L(\ell_q(I), Y)$ satisfying $\|\tilde u\|\leq C\|u\|$.
\end{itemize}

We will show that this happens if and only if $K_0(Z,Y)$ is
finite. To this end let us make explicit the following fact,
obtained during the proof of Theorem~\ref{main}: there is a
constant $N=N(X,Y,\Phi)$ such that, for each quasilinear map $F:X\to
Y$, there exist $u\in \mathfrak L(K,Y)$, with $\|u\|\leq NQ(F)$,
and a linear map $\ell:X\to Y$ satisfying
$$
\|F - (u\circ\Phi+\ell) \|\leq NQ(F).
$$

On the other hand, identifying $\ell_q(I)$ with $K\oplus_\Phi X$,
we see that given $u\in \mathfrak L(K,Y)$, any (not necessarily
continuous) linear extension of $u$ to $K\oplus_\Phi X$ has the
form $v(k,x)= u(k)-\ell'(x)$, where $\ell':X\to Y$ is a linear
map. Since
$$
\|v\|=\sup\frac{\|u(k)-\ell'(x)\|_Y}{\|(k,x)\|_\Phi}= \sup\frac{\|u(k)-u(\Phi(x))+u(\Phi(x))-\ell'(x)\|_Y}{\|(k-\Phi(x)\|_K+\|x\|_X},
$$
it follows that
$$
\| u\circ\Phi - \ell' \| \leq \|v\|\leq \Delta_Y(\|u\|+
\| u\circ\Phi - \ell' \|).
$$
Now, suppose $K_0(X,Y)$ is finite and let $K_0^+$ be any number greater than $K_0(X,Y)$.

Suppose that $u\in
\mathfrak L(K,Y)$ admits an extension to $\ell_q(I)$. There is no
loss of generality in assuming $\|u\|\leq 1$, so that
$Q(u\circ\Phi)\leq Q(\Phi)$. Let $\ell:Z\to Y$ be a linear map at
finite distance from $u\circ\Phi$, so that $u\circ\Phi-\ell$ is
bounded. As $Q(u\circ\Phi-\ell)=Q(u\circ\Phi)\leq Q(\Phi)$ there
is $\ell'\in \mathfrak L(X,Y)$ such that $\| u\circ\Phi-\ell -
\ell'\| = \| u\circ\Phi - (\ell+\ell')\|\leq Q(\Phi)K_0^+$. The
operator $v(k,x)=u(k)-(\ell+\ell')(x)$ is then an extension of $u$
to $K\oplus_\Phi X$, with
$$
\|v\|\leq \Delta_Y(\|u\|+ \| u\circ\Phi -  (\ell+\ell')\|\leq
\Delta_Y(1+ Q(\Phi)K_0^+).
$$
Hence $\varkappa^*$ is relatively open, which proves the ``if''
part.

As for the converse, suppose $(\bullet)$ holds and let
$B:X\to Y$ be bounded, with $Q(B)\leq 1$. Choose $u\in \mathfrak
L(K,Y)$ and a linear $\ell:X\to Y$ such that $\|u\|\leq N$ and
$$
\| B - (u\circ\Phi+\ell)\| \leq N.
$$
The operator $u$ obviously extends to $K\oplus_\Phi X$ and the
hypothesis yields an extension $v\in \mathfrak L(K\oplus_\Phi
X,Y)$ with $\|v\|\leq C\|u\|\leq CN$. As mentioned before such
$v$ has the form $v(k,x)= u(k)-\ell'(x)$, where $\ell':X\to Y$ is
a linear map and
$$
\| u\circ\Phi - \ell' \| \leq \|v\|\leq CN.
$$
Hence
$$
\|B - (\ell+\ell')\| \leq \Delta_Y( \|B - (u\circ\Phi+\ell)\| + \|
u\circ\Phi +\ell - (\ell'+\ell)\| \leq \Delta_YN(1+C).
$$
This completes the proof.
\end{proof}

\section{Counterexamples}\label{count}

Let us recall that a Banach space is said to be weakly compactly
generated (WCG) if it contains a weakly compact subset whose
linear span is dense in the whole space. Obviously each separable
Banach space is WCG and so are all reflexive spaces, in particular
$\ell_p(I)$ for $1<p<\infty$ and every index set $I$. Also,
$c_0(I)$ is WCG for every $I$, while $\ell_1(I)$ is WCG if and
only if $I$ is countable. A good general reference is
\cite{zizler}, in particular Section 3.

A quasi-Banach space $X$ is called a $\mathcal K$-space if
$K(X,\K)<+\infty$; equivalently, if every quasilinear function
$X\to\K$ is trivial; or else, if every short exact sequence
$$
0\longrightarrow \K \longrightarrow  Z\longrightarrow  X\longrightarrow  0
$$
splits. Banach spaces having nontrivial type $p>1$ are $\mathcal
K$-spaces (cf. \cite[Theorem 2.6]{kalton78}) as well as all
$\mathcal L_\infty$-spaces and their quotients \cite[Theorem
6.5]{kr}. As for specific examples, if $I$ is any infinite set,
then $c_0(I)$ is a $\mathcal K$-space, and $\ell_p(I)$ is a
$\mathcal K$-space if and only if $p\neq 1$, where
$p\in(0,\infty]$ (cf. \cite[Theorem 3.5]{kalton78}).

\begin{example}\label{infty}
Let $X$ be a WCG Banach space.
\begin{enumerate}
\item[(a)]If $X$ is nonseparable, then $K(X,c_0)=\infty$.
\item[(b)] If $X$ is a $\mathcal K$-space, then
$K_0(X, c_0)<\infty$.
\end{enumerate}

Therefore, if $X$ is a nonseparable WCG Banach $\mathcal K$-space, then $K_0(X,c_0)$ is finite but $K(X,c_0)$ is not.
\end{example}

\begin{proof} (a) It is shown in  \cite[Theorem
3.4]{cgpy} that if $X$ is a nonseparable WCG Banach space then
there is a nontrivial twisted sum of $c_0$ and $X$. Hence (see
Section~\ref{ts}) nontrivial quasilinear maps $F:X\to c_0$ exist
and so $K(X, c_0)=\infty$ ---when $X$ is either $c_0(I)$ or
$\ell^2(I)$ for uncountable $I$ the result goes back to \cite{jl}.
For a remarkably simple proof based on Parovichenko's theorem see
the recent paper \cite{y}. Different constructions can also be found in \cite{mar}.

(b)
 We prove now that
$K_0(X,c_0)$ is finite when $X$ is a WCG Banach $\mathcal
K$-space. Let $F:X\to c_0$ be a homogeneous bounded map, with
$Q(F)\leq 1$. The idea is to form the twisted sum $c_0\oplus_F X$
and obtain a projection $p$ onto $c_0$ having small norm. (Recall
from Section~\ref{ts} that in this case there is a linear map
$L:X\to c_0$ such that $\|F-L\|\leq \|p\|$). It is a classical
result in Banach space theory that $c_0$ is complemented by a
projection of norm at most 2 in any WCG Banach space. This is
Rosenthal's improvement of Veech's proof of Sobczyk's theorem
\cite{veech}. Of course $c_0\oplus_F X$ is WCG since $F$ is
trivial and so it is isomorphic to $c_0\oplus X$. The problem here
is that even if $c_0\oplus_F X$ is  isomorphic to a Banach space,
the functional $\|\cdot\|_F$ is only a quasi-norm and we need a
true norm to control the norm of the projection. This is the point
where the hypothesis that $X$ is a $\mathcal K$-space enters.

The details are as follows. Let $K=K(X,\mathbb K)$ be the $\mathcal K$-space constant of $X$. Then, for every homogeneous $f:X\to\K$ with $Q(f)\leq 1$ there is a linear map $\ell:X\to\K$ such that $|f(x)-\ell(x)|\leq K\|x\|$ for all $x\in X$. So, for finitely many points of $X$ we have,
\begin{align*}
\left|f\left(\sum_{i=1}^nx_i\right)-\sum_{i=1}^nf(x_i)\right|&=\left|f\left(\sum_{i=1}^nx_i\right)-\ell\left(\sum_{i=1}^nx_i\right)+
\sum_{i=1}^n\ell(x_i)-\sum_{i=1}^nf(x_i)\right|\\
&\leq K\left(\left\| \sum_{i=1}^nx_i \right\|+  \sum_{i=1}^n\|x_i\|\right)\\
&\leq 2K\left(\sum_{i=1}^n\|x_i\|\right).
\end{align*}
It follows that if $F:X\to c_0$ is quasilinear, with $Q(F)\leq 1$, then
\begin{equation*}
\left\|F\left(\sum_{i=1}^n
x_i\right)-\sum_{i=1}^n F(x_i)\right\|\leq 2K\sum_{i=1}^n\|x_i\|.
\end{equation*}
Hence, for every $n\in\N$ and $x_i\in X$,
\begin{align*}
\left\|\sum_{i=1}^n(y_i,x_i)\right\|_F&\leq \left\|\sum_{i=1}^ny_i -\sum_{i=1}^nF(x_i)\right\|_Y+
\left\| \sum_{i=1}^nF(x_i)  - F\left(\sum_{i=1}^n
x_i\right) \right\|_Y+ \sum_{i=1}^n\|x_i\|\\
&\leq (2K+1) \sum_{i=1}^n\|(y_i,x_i)\|_F.
\end{align*}
We define a norm on $c_0\oplus_F X$ by letting
$$
|(y,x)|_{\co}= \inf\left\{  \sum_{i=1}^n     \|(y_i,x_i)\|_F    :
(y,x)=  \sum_{i=1}^n (y_i,x_i)\right\}.
$$
By the preceding inequality,
$$
|(y,x)|_{\co}\leq \|(y,x)\|_F\leq (2K+1)|(y,x)|_{\co}.
$$
To conclude, consider the natural inclusion map $\imath: c_0\to
(c_0\oplus_F X, |\cdot|_{\co})$ given by $\imath(y)=(y,0)$. We
have that $\imath(c_0)$ is a separable subspace of $(c_0\oplus_F
X, |\cdot|_{\co})$, a WCG Banach space. Since
$$
\|y\|_\infty=\|(y,0)\|_F\leq (2K+1)|(y,0)|_{\co}
$$
the map $(y,0)\mapsto y$ has norm at most $(2K+1)$ from $(\imath(c_0),|\cdot|_{\co})$ to $c_0$ and by Veech's result it extends to a ``projection'' $p$ of
$ (c_0\oplus_F X, |(\cdot,\cdot)|_{\co})$ onto $c_0$ having whose norm is at most $2(2K+1)$. It is clear that the norm of $p$ as an operator from the quasi-Banach space $c_0\oplus_F X$ with its original norm $\|(\cdot,\cdot)\|_F$ to $c_0$ is at most  $2(2K+1)$ too. Writing $p(y,x)=y-L(x)$ as in Section~\ref{ts} we conclude that $\|F-L\|\leq   2(2K+1)$.\end{proof}

\section{``Positive'' results}

In most cases, however, $K_0(X,Y)$ is finite (if and) only if $K(X,Y)$ is.
Let us recall a few basic facts about ultrapowers of
quasi-Banach space spaces. Given a quasi-Banach space $Y$, a set
of indices $I$ and an ultrafilter $\mathscr{U}$ on $I$ one
considers the space of bounded families
$
\ell_\infty(I,Y)
$
with the sup quasi-norm and the subspace $N_\mathscr U=
\{(y_i): \lim_i\|y_i\|=0 \text{ along }\mathscr U\}$. The
ultrapower of $Y$ with respect to $\mathscr U$, denoted
$Y_\mathscr U$, is defined as the quotient space $
\ell_\infty(I,Y)/N_\mathscr U$ with the quotient quasi-norm. Let
$q_\U$ be the quotient map. As long as the quasi-norm of $Y$ is
continuous the quasi-norm in the ultrapower can be computed with
the formula
$$
\|[(y_i)]_U\|= \lim_U \|y_i\|,
$$
where we set $[(y_i)]_U =q_\U((y_i))$. In this case the diagonal map $\delta$
sending $y\in Y$ to the class of the contant family $(y)$ in $Y_\U$ is an isometric embedding.

An ultrasummand is a quasi-Banach space which appears complemented
in all its ultrapowers through the diagonal embedding. For Banach
spaces this is equivalent to ``it is complemented in its bidual
(or in any other dual Banach space)".
Typical nonlocally convex ultrasummands are $\ell_p(I)$ and the Hardy classes $H_p$ for $p\in(0,1)$.

A quasi-Banach space $X$ has
the Bounded Approximation Property (BAP) if there is a uniformly
bounded net of finite rank operators converging pointwise to the
identity on $X$. Obviously such a $X$ has separating dual (meaning
that for every nonzero $x\in X$ there is $x^*\in X^*$ such that
$x^*(x)\neq 0$). The BAP is equivalent to the following statement:
there is a constant $\lambda$ such that, whenever $E$ is a
finite-dimensional subspace of $X$, there is a finite rank
operator $u_E$ such that $u_E(x)=x$ for every $x\in E$
and $\|u\|\leq \lambda$. Every quasi-Banach space with a Schauder basis has the BAP.

\begin{theorem}\label{bap} \;
\begin{itemize}
\item[(a)] Let $X$ be a quasi-Banach space with the BAP and $Y$
an ultrasummand. If $K_0(X,Y)$ is finite, then so is $K(X,Y)$.

\item[(b)] Suppose $X$ is a $\mathcal K$-space and $Y$ an ultrasummand
with the BAP.  If $K_0(X,Y)$ is finite, then so is $K(X,Y)$.
\end{itemize}
\end{theorem}

\begin{proof}
By the Aoki-Rolewicz theorem we may assume the quasi-norm of $Y$ is continuous.

 To prove (a), let $\mathscr F$ denote the family of all finite dimensional
subspaces of $X$ ordered by inclusion. Let $\mathscr{U}$ be any
ultrafilter refining the Fr\'echet (order) filter on
$\mathscr{F}$ and let us consider the corresponding ultrapower
$Y_\mathscr{U}$. Since $Y$ is an ultrasummand, there is a bounded
linear projection $p: Y_\mathscr{U}\to Y$. Since $X$ has the BAP,
for each $E\in \mathscr F$ we may fix a finite rank $u_E\in \mathfrak L(X)$ fixing $E$ with $\|u_E\|\leq \lambda$, where $\lambda$ is the ``BAP constant'' of $X$. To
conclude these prolegomena, let $1_E: X \to X $ denote the
characteristic function of $E$; i.e., $1_E(x)=1$ if $x\in E$ and
$1_E(x)=0$ otherwise.

Now, let $F:X\to Y$ be quasilinear, with $Q(F)\leq 1$. We
define a mapping $G: X\to \ell_\infty(\mathscr F,Y)$ as
$$G(x)=(1_E(x)F(u_E(x)))_{E\in \mathscr F}.$$
(Observe that the family is bounded.)
 When $x\in E$ the $E$-th coordinate of
$G(x)$ is $F(x)$. Since for fixed $x\in X$ the set $\{E\in\mathscr F:x\in E\}$ belongs to $\mathscr U$, it follows that $q_\U \circ G = \delta\circ  F$. Consequently, $F=p\circ q_\U \circ G$. The map
$F\circ u_E:X\to Y$ is quasilinear and bounded because quasilinear maps
are bounded on finite dimensional spaces. Thus there is $\ell_E:
X\to Y$ such that
$$\|\ell_E-F\circ u_E\| \leq K_0^+Q(F\circ u_E) \leq   K_0^+Q(F)\|u_E\|\leq K_0^+\lambda,$$
provided $K_0^+>K_0(X,Y)$.
This allows us to define a (probably nonlinear) map $\ell: X \to \ell_\infty(\mathscr F,Y)$ by
$$\ell(x)=(1_E(x)\ell_E(x))_{E\in \mathscr F}.$$
The definition makes sense because the family $(1_E(x)\ell_E(x))_{E\in \mathscr F}$ is bounded for each $x\in X$. In fact, one has
$$
\|G(x)-\ell(x)\|_\infty= \sup_{E\in\mathscr F}\| 1_E(x)(F(u_E(x))-\ell_E(x))   \|_Y\leq K_0^+\lambda\|x\|.
$$
Hence if we put $L=q_\mathscr U \circ\ell$, then $\|q_\mathscr U \circ G-L\|\leq\|G-\ell\|\leq  K_0^+\lambda$.

Now, the point is that $L$ is linear: it is obviously homogeneous and moreover, given $x,y\in X$, the set $\{E\in\mathscr F: x,y\in E\}$ belongs to $\mathscr U$. For these $E$ one has $1_E(x+y)\ell_E(x+y)=1_E(x)\ell_E(x)+1_E(y)\ell_E(y)$ and so
$$
L(x+y)= [(1_E(x)\ell_E(x))_E]_\mathscr U  + [(1_E(y)\ell_E(y))_E]_\mathscr U=L(x)+L(y).
$$
Therefore $p\circ L: X\to Y$ is a linear map and
$$\|F-p\circ L\|=\|p\circ q_\U \circ G - p\circ L\| \leq \|p\|K_0^+\lambda.$$
Thus, every quasilinear map from $X$ to $Y$ is trivial and $K(X,Y)$ is finite, by \cite[Proposition 3.3]{kalton78}).

The proof of (b) is analogous. This time the index set is the family of all finite dimensional subspaces of $Y$
(instead of $X$) which we denote again by $\mathscr F$. For each $E\in\mathscr F$ we take a finite rank
operator $u_E\in \mathfrak L(Y)$ such that $u_E(y)=y$ for $y\in
E$, with $\|u_E\|\leq \lambda$.

Suppose $F:X\to Y$ is quasilinear, with $Q(F)\leq 1$. As before, we fix $K_0^+>K_0(X,Y)$. For each $E\in\mathscr F$, consider the composition $u_E\circ F$ as a map from $X$ to $u_E(Y)$, a finite dimensional space. As $X$ is a $\mathcal K$-space there is a linear
map $\ell_E: X\to u_E(Y)$ at finite distance from $u_E\circ F$.
Considering  now the difference $u_E\circ F-\ell_E$ as a  bounded homogeneous map from $X$ to $Y$ we have $Q(u_E\circ
F-\ell_E)=Q(u_E\circ F)\leq \lambda$ and so there is a linear
map $\ell'_E: X\to Y$ such that
$$
\|(u_E\circ F-\ell_E)-\ell'_E \|= \|u_E\circ F-(\ell_E+\ell'_E )\|\leq \lambda K_0^+
$$
Taking $L_E=\ell_E+\ell'_E$ we define a mapping $L: X\to \ell_\infty(\mathscr F,Y)$
as
$$L(x)= (1_E(F(x))L_E(x)).$$
The map is well defined (i.e., the family is bounded) since when
$F(x)\in E$, then $u_E\circ F(x)=F(x)$ and $\|L_E(x)-F(x)\|\leq
K_0^+\lambda \|x\|$; hence
$$
\sup_{E\in\mathscr F}\| 1_E(F(x))L_E(x) \|\leq \Delta_Y\left
(\|F(x)\|+K_0^+\lambda \|x\|\right ).
$$

Let $\mathscr U$ be an ultrafilter refining the Fr\'echet filter
on $\mathscr F$. Let us check that, against intuition, $q_\U\circ L:
X\to Y_\mathscr U$ is linear. It is obviously homogeneous; and it
is also additive: indeed, if $x,y\in X$, then as long as $E$
contains $F(x), F(y)$ and $F(x+y)$ one has
$$
1_E(F(x+y))L_E(x+y)=1_E(F(x))L_E(x)+1_E(F(y))L_E(y)
$$
which yields $q_\U (L(x+y))=q_\U (L(x)) + q_\U (L (y))$. Finally, if
$p:Y_\mathscr U\to Y$ is a bounded projection then $p\circ q_\U\circ L:
X\to Y$ is a linear map at finite distance from $F$: observe that
$u_E(F(x))=F(x)$ as long as $F(x)\in E$, so $\delta (F (x)) = [
(1_E(F(x)) u_E(F(x)))]_\U$. One then has, $\|\delta\circ F- L\| \leq K_0^+\lambda$, and thus
$$\|F-p\circ L\|=
\|p\circ\delta \circ F-p\circ L\|\leq \|p\| \leq K_0^+\lambda.$$
This shows that $F$ is trivial, and since $F$ is arbitrary we conclude that $K(X,Y)$ is finite.
\end{proof}

$\bigstar$ Theorems~\ref{criterion} and \ref{bap} together
indicate that for a wide class of Banach and quasi-Banach spaces
$\Ext(X,Y)$ is Hausdorff only if it is zero. Thus, for instance
$\Ext(X,Y)$ fails to be Hausdorff if $X$ and $Y$ are super
reflexive and one of them has a basis.

$\bigstar$ The spaces appearing in Example~\ref{infty} are rather
specific. It should be interesting to have other examples of
Banach spaces for which $\Ext(X,Y)$ is nonzero and Hausdorff.
Having separable counterexamples should be interesting, even
allowing Banach space extensions only.
To be more precise, if $X$ and $Y$ are Banach spaces, let us denote by
$\Ext_{\bf B}(X,Y)$ the space of extensions $0\to Y\to Z\to X\to 0$ in which $Z$ is a Banach space, where we identify equivalent extensions.

The simplest way to describe the ``natural'' topology in
$\Ext_{\bf B}(X,Y)$ (which in general does not agree with the
restriction of the topology in $\Ext(X,Y)$ is to take a projective
presentation of $X$ in the category $\bf B$ of Banach spaces and
operators, say
$$
0\longrightarrow K\longrightarrow \ell_1(I)\stackrel{\vp}\longrightarrow X\longrightarrow 0,
$$
with $K=\ker \vp$ and to consider $$
\Ext_{\bf B}^\varpi(X,Y)= \frac{\mathfrak L(K, Y)}{\varpi^* (\mathfrak
L(\ell_1(I), Y))}
$$
with the obvious seminorm.

$\bigstar$ Let $\ell_1\to L_1$ be a quotient map and $K$ its
kernel. The space $\Ext_{\bf B}^\vp(L_1, K)$ is obviously nonzero.
Is it Hausdorff in its natural topology? Needless to say the space
$\Ext(L_1, K)$ (of extensions of quasi-Banach spaces) is not
Hausdorff as it contains a complemented copy of $\Ext(L_1, \K)$,
which is nonzero \cite{kalton78,ribe} and fails to be Hausdorff in
view of Proposition~\ref{criterion} and Theorem~\ref{bap}(a).

$\bigstar$ Kalton and Ostrovskii asked (see the comments preceding
Theorem 3.7 in \cite{kaltostro}) if $K_0(X,\K)<\infty$ implies
$K(X,\K)<\infty$ for all Banach spaces $X$. They state without
proof that the answer is affirmative when $X$ has the BAP.
Theorem~\ref{bap}(a) is of course stronger.

Theorem~\ref{criterion} shows that in the quasi-Banach setting the
answer to Kalton-Ostrovski question is no. Indeed, it is shown in
\cite[Corollary 4]{cabecastlong} that if $Z$ is a quasi-Banach
$\mathcal K$-space with trivial dual and $Y$ is a subspace of $Z$,
then for $X=Z/Y$ one has $Y^*=\Ext(X,\K)$, up to linear
homeomorphism. In particular, for $0<p<1$ one has that
$\Ext(L_p/\K,\K) = \K$ is nonzero, hence $K(L_p/\K,\K)=\infty$,
and Hausdorff, hence $K_0(L_p/\K,\K)<\infty$.

$\bigstar$ Concerning Theorem~\ref{bap} (b) it is worth noting
that if $X$ fails to be a $\mathcal K$-space, then $\Ext(X,Y)\neq
0$ for all quasi-Banach spaces $Y$ having nontrivial dual (a
property implied by the BAP).

$\bigstar$ We do not know if the condition ``$Q([F])=0$ for
\emph{all} quasilinear maps $F:X\to Y$" implies $\Ext(X,Y)=0$. Of
course it is possible to have $Q[F]=0$ for \emph{some} nontrivial
$F$. 

\section{Stability constants}\label{affine}
In this final Section we explain why Example~\ref{infty} solves the main problem raised in \cite{lp}.
Let $D$ be a convex set in a linear space $X$ and $Y$ a Banach space. A function $f:D\to Y$ is said to be $\delta$-affine if
$$
\|f(tx+(1-t)y)-tf(x)-(1-t)f(y)\|_Y\leq \delta
$$
for every $x,y\in D$ and every $t\in[0,1]$. If the preceding inequality holds merely for $t={1\over 2}$ we say that $f$ is $\delta$-Jensen.
Of course $0$-affine functions are just the popular affine functions and $0$-Jensen functions are the so called Jensen (or midpoint affine) functions.

Following \cite{lp} we define $A_0(D,Y)$ as the infimum of those constants $C$ for which the following statement holds: if $f:D\to Y$ is a \emph{uniformly bounded} $\delta$-affine function, then there is an affine function $a:D\to Y$ such that $\|f(x)-a(x)\|_Y\leq C\delta$. We define $A(D,Y)$ as before, but omitting the words ``uniformly bounded''. Also, we define $J_0(D,Y)$ and $J(D,Y)$ by replacing ``affine'' by ``Jensen'' everywhere in the definition of $A_0(D,Y)$ and $A(D,Y)$, respectively.

Laczkovich and Paulin proved in \cite[Theorem 2.1]{lp} that
$$
A(D,Y)\leq J(D,Y)\leq 2A(D,Y).
$$
On the other hand, it trivially follows from \cite[Lemma 2.2(iii)]{lp} that every uniformly bounded $\delta$-Jensen function is $2\delta$-affine and, therefore,
$$
A_0(D,Y)\leq J_0(D,Y)\leq 2A_0(D,Y).
$$
The ``second Whitney constant for bounded functions'' $w_2^Y(D)$ also appears in \cite{lp}. We will not use it in this paper since it is clear from the discussion preceding Theorem 3.1 in \cite{lp} that $w_2^Y(D)={1\over 2}J_0(D,Y)$. We refer the reader to \cite{brudkalt} for a comprehensive study of the Whitney constants of the balls of various Banach and quasi-Banach spaces.

\begin{proposition}
Let $X$ and $Y$ be Banach spaces. If $B$ is the unit ball of $X$, then:
\begin{itemize}
\item[(a)] $K(X,Y)\leq A(B,Y)\leq 3+6K(X,Y)$.
\item[(b)] $K_0(X,Y)\leq A_0(B,Y)\leq 3+6K_0(X,Y)$.
\end{itemize}
\end{proposition}

\begin{proof} We write the proof of (a). The proof of (b) is the same, adding the word ``bounded'' from time to time to the text.

Let us check the first inequality. Suppose $F:X\to Y$ is quasilinear, with $Q(F)= 1$. Then the restriction of $F$ to $B$ is 1-affine. Indeed, taking $x,y\in B$ and $t\in [0,1]$ we have
\begin{align*}
\|F(tx+(1-t)y)-tF(x)-(1-t)F(y)\|&=\|F(tx+(1-t)y)-F(tx)-F((1-t)y)\|\\
&\leq \|tx\|+\|(1-t)y\|\leq 1.
\end{align*}
If $A^+$ is any number greater than $A(B,Y)$, then there is an affine function $a:B\to Y$ such that $\|F(x)-a(x)\|\leq A^+$ for every $x\in B$. It is obvious that there is a (unique) linear map $L:X\to Y$ such that $L(x)={1\over 2}(a(x)-a(-x))$ for $x\in B$. But $F$ is odd and so $\|F(x)-L(x)\|\leq A^+$ for every $x\in B$. Now, by homogeneity, we have
$$
\|F(x)-L(x)\|\leq A^+\|x\|\quad\quad(x\in X)
$$
and, therefore, $K(X,Y)\leq A(B,Y)$.

As for the second inequality in (a) let $f:B\to Y$ be 1-affine. Without no loss of generality we may assume $f(0)=0$. Put
$$
f_1(x)=\frac{f(x)-f(-x)}{2}\quad\quad\text{and}\quad\quad
f_2(x)=\frac{f(x)+f(-x)}{2}.
$$
Clearly, $f=f_1+f_2$. Moreover, $f_1$ is $1$-affine on $B$ and $\|f_2(x)\|\leq 1$ for every $x\in B$. Let $\Lambda$ be the set of lines of $X$ passing through the origin. Given $\lambda\in \Lambda$, let $f_\lambda$ denote the restriction of $f_1$ to $\lambda\cap B$, a set which is affinely equivalent to $[0,1]$. As $f_\lambda$ is 1-affine and $A([0,1],Y)\leq 2$ by \cite[Lemma 2.2(iii)]{lp} there is a linear map $f_\lambda^*:\lambda\to Y$ such that $\|f_1(x)-f_\lambda^*(x)\|_Y\leq 2$ for $x\in\lambda\cap B$.

Define $f^*:X\to Y$ taking
$
f^*(x)=f^*_\lambda(x)$ 
 for $x\in\lambda$.
Then $f^*$ is homogeneous on $X$ and $3$-affine on $B$. Thus, for $x,y\in B$ we have
$$
\|f^*(x+y)-f^*(x)-f^*(y)\|= 2 \left\|f^*\left(\frac{x+y}{2}\right)-\frac{f^*(x)}{2}-\frac{f^*(y)}{2}\right\|\leq 6.
$$
And, by homogeneity, for arbitrary $x,y\in X$,
$$
\|f^*(x+y)-f^*(x)-f^*(y)\|\leq 6\max(\|x\|,\|y\|)\leq 6(\|x\|+\|y\|),
$$
that is, $Q(f^*)\leq 6$.

Now, if $K^+$ is a number greater than $K(X,Y)$, there is a linear map $L:X\to Y$ such that $\|f^*(x)-L(x)\|\leq 6K^+\|x\|$ for all $x\in X$. Hence for $x\in B$ we have
$$
\|f(x)-L(x)\|\leq \|f(x)-f_1(x)\|+\|f_1(x)-f^*(x)\|+\|f^*(x)-L(x)\|\leq 1+2+6K^+
$$
and we are done.
\end{proof}

\begin{corollary}
Let $B$ be the unit ball of a nonseparable WCG Banach $\mathcal
K$-space. Then $A_0(B,c_0)$ is finite while $A(B,c_0)$ is
not.\hfill$\square$
\end{corollary}


\begin{thebibliography}{99}




\bibitem{brudkalt} Y.A. Brudnyi and N. J. Kalton, \emph{Polynomial
approximation on convex subsets of $\R^n$}, Constr. Approx. 16
(2000) 161-199.






\bibitem{cabecastlong}
F. Cabello S\'{a}nchez and J. M. F. Castillo, \emph{The long
homology sequence for quasi-Banach spaces, with applications},
Positivity 8 (2004) 379--394.







\bibitem{castgonz}
J. M. F. Castillo and M. Gonz\'{a}lez, \emph{Three-space problems
in Banach space theory,} Lecture Notes in Mathematics 1667,
Springer-Verlag (1977).

\bibitem{cgpy}
J.M.F. Castillo, M. Gonz\'{a}lez, A. Plichko and D. Yost,
\emph{Twisted propertes of Banach spaces}, Math. Scand. 89 (2001),
217--244




\bibitem{jl}
W. B. Johnson and J. Lindenstrauss, \emph{Some remarks on weakly
compactly generated Banach spaces}, Israel J. Math. 17 (1974),
219--230. Correction ibid 32 No. 4 (1979), 382--383.



\bibitem{kalton78} N. J. Kalton, \emph{The three-space problem for locally
bounded F-spaces}, Compo. Math. 37 (1978) 243--276.




\bibitem{kaltostro}
N. J. Kalton and M. Ostrovskii, \emph{Distances between Banach
spaces}. Forum Math.  11  (1999), no. 1, 17--48


\bibitem{kaltpeck}
N. J. Kalton and N. T. Peck, \emph{Twisted sums of sequence spaces
and the three space problem}, Trans. Amer. Math. Soc. 255 (1979)
1--30.

\bibitem{kpr}
N.J. Kalton, N.T. Peck and W. Roberts, \emph{An F-space sampler,}
London Mathematical Society Lecture Note Series 89.


\bibitem{kr}
N. J. Kalton and J. W. Roberts, \emph{Uniformly exhaustive
submeasures and nearly additive set functions}, Trans. Amer. Math.
Soc. 278 (1983) 803-816.



\bibitem{lp}
M. Laczkovich and R. Paulin, \emph{Stability constants in linear
spaces}, Constr. Approx. 34 (2011) 89--106.


\bibitem{maclane}
S. Mac Lane,
\emph{Homology,}
Grundlehren der mathematischen Wissenschaften 114, Springer-Verlag 1975.

\bibitem{mar} W. Marciszewski, R. Pol, \emph{On Banach spaces whose norm-open sets are $F_\sigma$-sets in the weak topology}, J. Math. Anal. and Appl. 350 (2009) 709-722.

\bibitem{ribe}
M. Ribe, \emph{Examples of the nonlocally convex three-space
problem}, Proc. Amer. Math. Soc. 73 (1979), 351-355.




\bibitem{r}
S Rolewicz, \emph{Metric linear spaces}. Second edition.
Mathematics and its Applications (East European Series), 20. D.
Reidel Publishing Co., Dordrecht; PWN--Polish Scientific
Publishers, Warsaw, 1985


\bibitem{stiles}
W. J. Stiles, \emph{Some properties of $\ell\sb{p}$, $0<p<1$},
Studia Math.  42  (1972), 109--119


\bibitem{veech}
W. A. Veech, \emph{Short proof of Sobczyk's theorem}, Proc. Amer.
Math. Soc.  28  1971 627--628









\bibitem{y}
D. T. Yost, \emph{A different Johnson-Lindenstrauss space}, New
Zealand J. Math. 37 (2008) 47--49.



\bibitem{zizler}
V. Zizler, \emph{Nonseparable Banach spaces}, in Handbook of the
geometry of Banach spaces, Vol. 2,  1743--1816, North-Holland,
Amsterdam, 2003


\end{thebibliography}
\end{document}